\documentclass[letterpaper,11pt]{article}

\usepackage{amsfonts,amsmath,amssymb,amsthm}

\usepackage{url}
\newtheorem{thm}{Theorem}[section]
\newtheorem{prop}[thm]{Proposition}
\newtheorem{lem}[thm]{Lemma}
\newtheorem{cor}[thm]{Corollary}

\newtheorem{asm}{Assumption}

\theoremstyle{remark}
\newtheorem{rem}[thm]{Remark}

\theoremstyle{definition}

\newcommand{\ra}{\rightarrow}

\newcommand{\N}{\mathbb N}     
\newcommand{\Q}{\mathbb Q}     
\newcommand{\R}{\mathbb R}     
\newcommand{\Z}{\mathbb Z}     
\newcommand{\Zp}{\Z_+}	       

\renewcommand{\a}{\alpha}
\renewcommand{\b}{\beta}

\renewcommand{\d}{\delta}

\newcommand{\e}{\varepsilon}

\renewcommand{\l}{\lambda}

\newcommand{\s}{\sigma}

\newcommand{\bigo}{\mathcal{O}}

\newcommand{\w}{\omega}              
\renewcommand{\P}{\mathbb{P}}        
\newcommand{\E}{\mathbb{E}}          
\newcommand{\vp}{\mathrm{v}_P}       
\newcommand{\iid}{i.i.d.\ }          
\newcommand{\indd}[1]{ \mathbf{1}\{ #1 \} }

\newcommand{\Cb}{\mathcal{C}_{0}} 	
\newcommand{\prof}{\alpha_0}		

\begin{document}


\title{Systems of One-dimensional Random Walks in a Common Random Environment}
\author{Jonathon Peterson
\thanks{Research partially supported by National Science Foundation grant DMS-0802942.
Part of this work was completed during a visit to the Institut Mittag-Leffler in Djursholm, Sweden for the program ``Discrete Probability.''}
 \\ Cornell University \\ Department of Mathematics \\ Malott Hall \\ Ithaca, NY 14853 \\ Email: \url{peterson@math.cornell.edu} \\ URL: \url{http://www.math.cornell.edu/~peterson}}

\date{}

\maketitle

\begin{abstract}
We consider a system of independent one-dimensional random walks in a common random environment under the condition that the random walks are transient with positive speed $\vp$. 
We give upper bounds on the quenched probability that at least one of the random walks started in the interval $[An, Bn]$ has traveled a distance of less than $(\vp - \e)n$. This leads to both a uniform law of large numbers and a hydrodynamic limit. 
We also identify a family of distributions on the configuration of particles (parameterized by particle density) which are stationary under the (quenched) dynamics of the random walks and show that these are the limiting distributions for the system when started from a certain natural collection of distributions.
\end{abstract}



Submitted to EJP on July 21, 2009, final version accepted June 22, 2010.

\newpage
\begin{section}{Introduction and Statement of the Main Results}

The object of study in this paper is a system of independent one-dimensional random walks in a common random environment. We modify the standard notion of random walks in random environment (RWRE) to allow for infinitely many particles. Let $\Omega := [0,1]^\Z$. An environment is an element $\w=\{\w_x\}_{x\in\Z} \in \Omega$. 
Given an environment $\w$, we let $\{X^{x,i}_\cdot \}_{x\in\Z, i\in \Z_+}$ be an independent collection of Markov chains with law $P_\w$ defined by
\[
 P_\w(X_0^{x,i} = x) = 1, 
\quad\text{and}\quad 
P_\w( X_{n+1}^{x,i} = z | X_n^{x,i} = y ) =
  \begin{cases}
   \w_y & z = y+1 \\
   1-\w_y & z= y-1 \\
   0 & \text{otherwise}.
  \end{cases}
\]
When we are only concerned with a single random walk started at $y\in\Z$ we will use the notation $X^{y}_{n}$ instead of $X^{y,1}_n$. Moreover, if the walk starts at the origin we will use the notation $X_n$ instead of $X^{0}_n$. 

The law $P_\w$ is called the \emph{quenched} law of the random walks. Let $P$ be a probability measure on $\Omega$. The \emph{averaged} law of the random walks is defined by averaging the quenched law over all environments. That is, $\P(\cdot) = \int_\Omega P_\w(\cdot) P(d\w)$. Quenched and averaged expectations will be denoted by $E_\w$ and $\E$, respectively, and expectations according the the measure $P$ on environments will be denoted by $E_P$. 

In this paper we will always make the following assumptions on the environment
\begin{asm}\label{IIDasm}
 The environments are uniformly elliptic and i.i.d. That is, the random variables $\{\w_x\}_{x\in\Z}$ are \iid under the measure $P$, and $P( \w_0 \in [c,1-c] ) = 1$ for some $c>0$. 
\end{asm}
\begin{asm}\label{Basm}
 $E_P[ \rho_0 ] < 1$, where $\rho_x := \frac{1-\w_x}{\w_x}$.
\end{asm}
Assumptions \ref{IIDasm} and \ref{Basm} imply that the random walks are transient to $+\infty$ with positive speed \cite{sRWRE}. That is, for any $x\in\Z$ and $i\geq 1$, 
\begin{equation}\label{LLN}
 \lim_{n\ra\infty} \frac{X^{x,i}_n-x}{n} = \frac{1- E_P[ \rho_0]}{1+E_P[\rho_0]} =: \vp, \quad \P-a.s.
\end{equation}
Our first main result is the following uniform version of \eqref{LLN}. 
\begin{thm}\label{ULLN}
Let Assumptions \ref{IIDasm} and \ref{Basm} hold. Then for any $A<B$ and $\gamma < \infty$,
\[
\lim_{n\ra\infty} \max_{y\in(An, Bn], \; i \leq n^\gamma} \left| \frac{X^{y,i}_n - y}{n} - \vp \right| = 0, \quad \P-a.s.
\]
\end{thm}
An obvious strategy for proving Theorem \ref{ULLN} would be to use a union bound and the fact that the probabilities $\P\left( \left| \frac{X^{y,i}_n - y}{n} - \vp \right| \geq \e \right)$ vanish for any $\e>0$ as $n\ra\infty$. However, if the distribution $P$ on environments is nestling (that is $P(\w_0 < \frac{1}{2})>0$), these probabilities only vanish polynomially fast (see \cite{dpzTE}) which is not good enough to prove Theorem \ref{ULLN}. The key to proving Theorem \ref{ULLN} is instead the following uniform analog of the quenched sub-exponential slowdown probabilities given in \cite{gzQSubexp}. 
\begin{prop}\label{UQLDPslow}
Let Assumptions \ref{IIDasm} and \ref{Basm} hold, and let $E_P \rho_0^s = 1$ for some $s> 1$. Then, for any $A<B$, $v\in(0,\vp)$, $\d>0$,  and $\gamma < \infty$,
 \begin{equation}\label{UXndecay}
  \limsup_{n\ra\infty} \frac{1}{n^{1-1/s-\d}} \log P_\w\left( \exists y \in(An, Bn], \; i\leq n^\gamma :\; X_{n}^{y,i} - y \leq nv \right) = -\infty, \qquad P-a.s.
 \end{equation}
\end{prop}
\begin{rem}
The proofs of the quenched \cite{gzQSubexp} and averaged \cite{dpzTE} subexponential rates of decay for slowdown probabilities make clear why there is a difference in the rate of decay. Slowdowns occur under the average probability by creating an atypical environment near the origin which traps the random walk for $n$ steps. For slowdowns under the quenched measure, an environment is fixed and (with high probability) contains only smaller traps - making it harder to slow down the random walk. 
Since particles starting at different (nearby) points in the same fixed environment encounter essentially the same traps, it is reasonable to expect that the uniform quenched large deviations decay at the same rate as the quenched large deviations of a single RWRE. 
\end{rem}

As an application of Theorem \ref{ULLN}, we prove a hydrodynamic limit for the system of random walks. 
For any $N$, let $\eta^N_0(\cdot) \in (\Zp)^\Z$ be an initial configuration of  particles. 
We will allow $\eta^N_0$ to be either deterministic or random (even depending on $\w$), 
but we still require that given $\w$, the paths of the random walks $\{ X^{x,i} \}_{x\in\Z,\; i\leq \eta^N_0(x)}$ are independent of the $\eta^N_0(x)$. 
As a slight abuse of notation we will use $P_\w$ and $\P$ to denote the expanded quenched and averaged probability measures of the systems of RWRE with (random) initial conditions $\eta^N_0$. 
Let
\begin{equation}\label{etat}
 \eta_n^N(x) = \sum_{y\in\Z}\sum_{i=1}^{\eta^N_0(y)} \indd{ X_n^{y,i} = x }
\end{equation}
be the number of particles at location $x$ at time $n$ when starting with initial configuration $\eta^N_0$. 
A hydrodynamic limit essentially says that if (when scaling space by $N$), the initial configurations $\eta^N_0$ are approximated by a bounded function $\prof(y)$, then (with space scaled by $N$) the configuration $\eta^N_{Nt}$ is approximated by $\prof(y-\vp t)$. 


\begin{thm}\label{strongweakhydro}
Let $\Cb$ be the collection of continuous functions with compact support on $\R$.
Assume the initial configurations $\eta^N_0$ are such that there exists a bounded function $\prof(\cdot)$ such that for all $g\in\Cb$,
\begin{equation} \label{hydroIC}
 \lim_{N\ra\infty} \frac{1}{N} \sum_{x \in \Z } \eta^N_{0}(x) g(x/N) = \int_\R \prof(y) g(y) dy, \quad \P-a.s.
\end{equation}
Then, for all $g\in\Cb$ and $t<\infty$,
\begin{equation} \label{hydroConc}
 \lim_{N\ra\infty} \frac{1}{N} \sum_{x \in \Z} \eta^N_{Nt}(x) g(x/N) = \int_\R \prof(y - \vp t) g(y) dy, \quad \P-a.s.
\end{equation}
Moreover, if instead we only assume that the convergence in \eqref{hydroIC} holds in $P_\w$-probability, then the conclusion \eqref{hydroConc} also holds in $P_\w$-probability as well. 
\end{thm}
\begin{rem}
 An example of where assumption \eqref{hydroIC} is satisfied is when $\{\eta^N_0(x)\}_{x\in\Z}$ are independent and $\eta_0^N(x) \sim \text{Poisson}( \prof(x/N) )$. 
\end{rem}

The hydrodynamic limits in Theorem \ref{strongweakhydro} describe the behavior of the system of RWRE when time and space are both scaled by $N$. If we do not rescale space, we can study the limiting distribution of the configuration of particles as the number of steps tends to infinity. 
That is, given the distribution of the initial configuration of particles $\eta_0 \in (\Zp)^\Z$, we identify the limiting distribution of the process $\eta_n$.
Before stating our last main result we need to specify the assumptions on the initial configurations. We will allow the initial distribution to depend on the environment $\w$ (in a measurable way), but given $\w$ we will require that the initial configuration is a product measure. To make this precise, let $\Upsilon$ be the space of probability distributions on the non-negative integers $\Zp$ equipped with the topology of weak-$*$ convergence (convergence in distribution), and let $\nu: \Omega \ra \Upsilon$ be a measurable function. 
Also, for any $x\in\Z$ let $\theta^x$ be the shift operator on environments defined by $(\theta^x\w)_y = \w_{x+y}$. 
Then, for each environment $\w$, let $\eta_0$ have distribution $\nu^\w := \bigotimes \nu(\theta^x\w)$. That is, given $\w$, $\{\eta_0(x)\}_{x\in\Z}$ is an independent family of random variables and $\eta_0(x)$ has distribution $\nu(\theta^x\w)$. Let $P_{\w, \nu^\w}$ denote the quenched distribution of the system of random walks with initial distribution given by $\nu^\w$, and let $\P_\nu(\cdot) = \int_{\Omega} P_{\w,\nu^\w}(\cdot) P(d\w)$ be the corresponding averaged distribution. The corresponding expectations are denoted by $E_{\w,\nu^\w}$ and $\E_\nu$, respectively. 

We now define the unique family of limiting distributions for initial configurations with distributions given by $\P_{\nu}$ for some $\nu$. 
For any $\a>0$, let $\pi_\a:\Omega\ra\Upsilon$ be defined by 
\begin{equation} \label{fdef}
 \pi_\a(\w) = \text{Poisson}(\a f(\w)), \quad\text{where}\quad 
f(\w) =
 \frac{1}{\w_0} \left(1 + \sum_{i=1}^\infty \prod_{j=1}^i \rho_j \right). 
\end{equation}
The formula for $\vp$ in \eqref{LLN} and the fact that $P$ is \iid imply that $E_P[ f(\w) ]=1 / \vp$, and therefore $\E_{\pi_\a}[ \eta_0(0) ] = E_P[ \a f(\w) ] = \a/\vp$. 
Our final main result is the following.
\begin{thm}\label{uniquestat}
 Let Assumptions \ref{IIDasm} and \ref{Basm} hold. If $\nu:\Omega\ra \Upsilon$ is such that $\E_\nu( \eta_0(0) ) < \infty$, then $\P_\nu(\eta_n \in \cdot )$ converges weakly to $\P_{\pi_\a}(\eta_0\in\cdot)$ (in the space of probability measures on $(\Zp)^\Z$) as $n\ra\infty$, with $\a = \vp \E_\nu( \eta_0(0) )$.
\end{thm}

The structure of the paper is as follows. Section \ref{UQLDP} is devoted to the proof of Proposition \ref{UQLDPslow}. The proof is an adaptation of the proof of the similar bounds given in \cite{gzQSubexp} for a single random walk. 
In Section \ref{Hydro} we give the proof of Theorem \ref{ULLN} and we show how it can be used to prove the hydrodynamic limits in Theorems \ref{strongweakhydro}. Finally, in Section \ref{StatDist} we prove Theorem \ref{uniquestat} via a coupling technique. 

\end{section}

\begin{section}{Uniform Quenched Large Deviations}\label{UQLDP}

In this Section, we will make the following assumption
\begin{asm}\label{Nestasm}
 $E_P [\rho_0^s] = 1$ for some $s>1$. 
\end{asm}

Quenched and averaged large deviation principles for a single RWRE are known in the setting we are considering. For speedup (that is, when $X_n \approx nv $ with $v>\vp$), both $\P(X_n > nv)$ and $P_\w(X_n > nv)$ decay exponentially fast \cite{cgzLDP} (although with different constants in the exponent). However, for slowdown (that is, when $X_n \approx nv$ with $v<\vp$) the averaged and quenched probabilities both decay sub-exponentially. In fact the averaged rate of decay is roughly $n^{1-s}$ and the quenched rate of decay is roughly $e^{-n^{1-1/s}}$. Precise statements (see \cite{dpzTE,gzQSubexp}) are, if $E_P[ \rho_0^s ] = 1$ for some $s>1$ then for any $v<\vp$
\begin{equation}\label{asubexp}
 \lim_{n\ra\infty} \frac{\log \P(X_n \leq nv)}{\log n} = 1-s, 
\end{equation}
and for any $\d>0$
\begin{equation}\label{qsubexp}
\liminf_{n\ra\infty} \frac{ \log P_\w(X_n < nv) }{n^{1-1/s+\d}} = 0,
\quad\text{and}\quad
 \limsup_{n\ra\infty} \frac{ \log P_\w(X_n < nv) }{n^{1-1/s-\d}} = -\infty,  \quad \P-a.s.
\end{equation}
A uniform analog (in the form of Proposition \ref{UQLDPslow}) of the first statement in \eqref{qsubexp} can be shown to hold quite easily.
Indeed,
\[
 P_\w\left( \exists y \in(An, Bn], i\leq n^\gamma : X_{n}^{y,i} - y < nv \right) \geq  P_\w(X_n^{Bn} - Bn < nv) = P_{\theta^{Bn}\w} (X_n < nv). 
\]
The proof of the quenched large deviation lower bounds in \cite{gzQSubexp} can then be repeated for the (deterministically) shifted environment $\theta^{Bn}\w$.

This section is devoted to the proof of Proposition \ref{UQLDPslow} which is the uniform analog of the second statement in \eqref{qsubexp}. 
We will prove Proposition \ref{UQLDPslow} by adapting the proof of quenched subexponential decay in \cite{gzQSubexp}. For clarity, we will divide the proof into a series of lemmas that progressively reduce the problem to an easier one. 
As was done in \cite{gzQSubexp}, we begin by reducing the proof of Proposition \ref{UQLDPslow} to the study of the large deviations of hitting times. 
For $x,y\in\Z$, let $T^{y}_x$ denote the amount of time it takes for the random walk started at $y$ to move a distance of $x$. That is,
\[
 T^{y}_x := \inf \{ n\geq 0: X_n^{y} = y+x \} \, .
\]
Moreover, it will be enough to prove large deviation upper bounds for hitting times on a dense enough subsequence of integers. We fix $\d>0$ for the remainder of the section and let $n_j:= \lfloor j^{2/\d} \rfloor$. 
\begin{lem}\label{Timelem}
Suppose that for any $A<B$ and any $\mu > \vp^{-1}$, 
\begin{equation} \label{subseqdecay}
 \limsup_{j\ra\infty} \frac{1}{n_j^{1-1/s-\d}} \log P_\w\left( \exists y \in(An_j, Bn_j]: T_{n_j}^{y} \geq n_j \mu \right) = -\infty, \quad P-a.s. 
\end{equation}
Then, for any $A<B$ and $v\in(0,\vp)$,
\begin{equation} \label{Xndecay}
 \limsup_{n\ra\infty} \frac{1}{n^{1-1/s-\d}} \log P_\w\left( \exists y \in(An, Bn] : X_{n}^{y,i} - y \leq nv \right) = -\infty, \quad P-a.s.
\end{equation}
Moreover, for any $\gamma < \infty$, \eqref{UXndecay} holds as well. 
\end{lem}
\begin{proof}
 The proof that \eqref{subseqdecay} implies \eqref{Xndecay} is essentially the same as the argument given on pages 181-182 in \cite{gzQSubexp}, and thus we only give a brief sketch. First, from the monotonicity of hitting times and the fact that $\lim_{j\ra\infty} n_{j+1}/n_j = 1$ we can deduce that \eqref{subseqdecay} implies an analogous statement for large deviations of $T_n$: (not along a subsequence). 
\begin{equation} \label{Tndecay}
 \limsup_{n\ra\infty} \frac{1}{n^{1-1/s-\d}} \log P_\w\left( \exists y \in(An, Bn] : T_{n}^{y} \geq n \mu \right) = -\infty, \quad P-a.s.,
\end{equation}
The passage from large deviations of $T_n$ to the statement \eqref{Xndecay} is accomplished by the fact that the amount a random walk backtracks has exponential tails. That is, there exist constants $C, \theta>0$ such that $\P( T_{-x} < \infty ) \leq C e^{-\theta x}$ for all $x \geq 1$. 

Finally, for any $\gamma<\infty$, by a union bound and the fact that $\log( n^\gamma) = o(n^{1-1/s-\d})$ we obtain that \eqref{Xndecay} implies \eqref{UXndecay}. 
\end{proof}

By the above lemma, we may reduce ourselves to proving \eqref{Tndecay} for a fixed $\d>0$, $A<B$ and $\mu > \vp^{-1}$. 
To this end, let $k_j = \lfloor n_j^{1/s+\d}/(1-\e) \rfloor $ for some small $\e>0$. We next divide the environment into blocks of length $k_j$. Let 
\[
 \mathcal{K}_j := k_j \Z = \{ m k_j: m\in \Z \}.
\]
The proof of \eqref{qsubexp} in \cite{gzQSubexp} was accomplished by studying the induced random walk on the lattice $\mathcal{K}_j$. The averaged large deviation estimates \eqref{asubexp} were used to analyze the tails of the amount of time for the original random walk to produce a step in the induced random walk. Also, the fact that backtracking probabilities decay exponentially in the distance backtracked was used to control the number of steps the induced random walk ever backtracked. Our strategy in proving \eqref{Tndecay} will be to adapt the techniques used in \cite{gzQSubexp} to 
study to multiple induced random walks started at different locations in $\mathcal{K}_j$. 

For a fixed $A<B$, let
\[
 \mathcal{M}_j := \{ m\in\Z : m k_j \in (An_j - k_j, Bn_j] \}. 
\]
Thus, for any $y \in (An_j, Bn_j]$, there exists a unique $m\in \mathcal{M}_j$ such that $m k_j \leq y < (m+1) k_j$. 
The next lemma reduces the proof of \eqref{Tndecay} to the study of random walks started at points in $\mathcal{M}_j$. 
\begin{lem}\label{latticelem}
Let $\e>0$, and let $k_j$ and $\mathcal{M}_j$ be defined as above. Then, 
\begin{align*}
 & P_\w\left( \exists y \in(An_j, Bn_j]: T_{n_j}^{y} \geq n_j \mu \right) \leq k_j P_\w( \exists m\in \mathcal{M}_j : T_{n_j+k_j}^{m k_j} \geq  n_j \mu ) .
\end{align*}
\end{lem}
\begin{proof}
 First note that
\[
 P_\w( \exists y \in(An_j, Bn_j]: T_{n_j}^{y} \geq n_j \mu ) 
 \leq \sum_{l=0}^{k_j-1} P_\w( \exists m\in \mathcal{M}_j: T_{n_j}^{m k_j + l} \geq n_j \mu ) 
\]
Now, one way for the event $\{T_{n_j+k_j}^{mk_j} \geq n_j \mu \}$ to occur is if the random walk starting at $m k_j$ after first hitting $m k_j + l$ then takes more than $n_j \mu $ steps to reach $m k_j + l + n_j < (m+1)k_j + n_j$. Thus, the strong Markov property implies that 
\[
 P_\w( \exists m\in \mathcal{M}_j: T_{n_j}^{m k_j + l} \geq n_j \mu ) \leq P_\w( \exists m\in \mathcal{M}_j : T_{n_j+k_j}^{mk_j} \geq n_j \mu ) .
\]
Since this last term does not depend on $l$, the proof of the lemma is finished. 
\end{proof}

The following lemma is the key step in the proof of Proposition \ref{UQLDPslow}. 
\begin{lem}\label{keylemma}
 Let $\vp^{-1} < \mu' < \mu$ and let $\e < \frac{\mu-\mu'}{3\mu}$. Then, there exists a $\theta>0$ such that $P-a.s.$ for any $A<B$ and all $j$ sufficiently large,
 \[
  \max_{m\in \mathcal{M}_j} P_\w( T^{m k_j}_{n_j + k_j} \geq n_j \mu ) \leq e^{-\theta \e n_j/2} + e^{-\l \e \mu n_j^{1-1/s-\d}}, \qquad \forall \l > 0.  
 \]
\end{lem}

\begin{proof}
For any $m \in\Z$, $j\geq 1$, and $i\geq 0$, define $\s^{j,m}(i)$ by 
\[
 \s^{j,m}(0) = 0, \qquad \s^{j,m}(i) = \inf \left\{ t > \s^{j,m}(i-1) : X^{m k_j}_t \in \mathcal{K}_j \backslash \{X^{m k_j}_{\s^{j,m}(i-1)} \} \right\}. 
\]
That is, $\s^{j,m}(i)$ are the times when the random walk started at $m k_j$ visits a point on the lattice $\mathcal{K}_j$ other than the one previously visited. 
Let $Y_i^{j,m} = \frac{1}{k_j}X^{m k_j}_{\s^{j,m}(i)} - m$ be the embeddeding of the random walk $X^{mk_j}_t$ on the lattice $\mathcal{K}_j$ re-centered to begin at the origin and scaled to have unit step sizes. 

Let $N_j := \lfloor n_j^{1-1/s-\d} \rfloor$. Then, if the random walk started at $m k_j$ has not gone $n_j + k_j$ steps to the right by time $n_j \mu$ then either the embedded random walk takes at least $N_j$ steps to move $n_j/k_j + 1$ steps to the right, or it takes more than $n_j \mu$ steps of the original random walk to record $N_j$ steps in the embedded random walk. Therefore, 
\begin{align}
 P_\w( T^{m k_j}_{n_j + k_j} \geq n_j \mu ) & \leq P_\w( \inf \{i: Y^{j,m}_i  = \lceil n_j /k_j \rceil + 1\}  > N_j ) + P_\w( \s^{j,m}(N_j) > n_j \mu ) \nonumber \\
 &\leq P_\w( Y^{j,m}_{N_j} < \lceil n_j /k_j \rceil + 1) + P_\w( \s^{j,m}(N_j) > n_j \mu ) \label{embedded}. 
\end{align}

Let $I_j := \mathcal{K}_j \cap ((A-1)n_j - k_j, (B+1)n_j]$ be the points of the lattice $\mathcal{K}_j$ that are possible to reach in $N_j$ steps of an embedded random walk started at a point in $\mathcal{M}_j$.  (Note that this definition of $I_j$ is different from the one in \cite{gzQSubexp}, but what is important is that $|\mathcal{I}_j| = \bigo( n_j/k_j )$ is still true). Then, it is possible to show (c.f. Lemma 6 in \cite{gzQSubexp}) that for any $\theta < -\frac{\log (E_P \rho)}{1-\e}$, $P-a.s.$ there exists a $J_1= J_1(\w,\theta, \e, \d)$ such that for all $j\geq J_1$
\[
 \max_{i\in I_j} P_\w ( T^{i k_j}_{-k_j} < T^{i k_j}_{k_j} ) \leq e^{-\theta n_j^{1/s+\d}}. 
\]
Let $S^{j,\theta}_i$ be a simple random walk with 
\[
 P( S_{i+1}^{j,\theta} = S_i^{j,\theta} + 1 | S_i^{j,\theta} ) = 1 - P( S_{i+1}^{j,\theta} = S_i^{j,\theta} - 1 | S_i^{j,\theta} ) = 1- e^{-\theta n_j^{1/s+\d}}. 
\]
Thus, for $j$ sufficiently large, $S_i^{j,\theta}$ is stochastically dominated by $Y^{j,m}_i$ for $i \leq N_j$. As in Lemma 9 of \cite{gzQSubexp}, large deviation estimates for the simple random walk $S_i^{j,\theta}$ can be used to show that for $\theta < -\frac{\log E_P \rho}{1-\e}$ and $j \geq J_1$, 
\begin{equation}\label{embedded1}
 P_\w( Y^{j,m}_{N_j } < \lceil n_j /k_j \rceil + 1) \leq e^{-\frac{\theta \e}{2} n_j}, \qquad \forall m\in\mathcal{M}_j. 
\end{equation}

A trivial modification of Lemmas 5 and 7 in \cite{gzQSubexp} provides an upper bound on the quenched tails of the amount of time it takes for a random walk starting at a point in $I_j$  to reach a neighboring point in $I_j$. These estimates are enough to imply that (cf. Lemma 8 in \cite{gzQSubexp}), $P-a.s.$, there exists a $J_0=J_0(\w, A, B, \e, \mu', \d)$ such that for all $j\geq J_0$, $m\in \mathcal{M}_j$ and $i\leq N_j$, 
\[
 E_\w e^{\l \s^{j,m}(i)/ k_j} \leq \left( e^{\l\mu'(1+\e)} + g_j(\l,\mu',\e, \d)  \right)^i, \qquad \forall \l>0,
\]
for some $g_j(\l,\mu',\e, \d) \ra 0$ as $j\ra\infty$. 
Therefore, for $j \geq J_0$, $m\in \mathcal{M}_j$, and $\l > 0$,
\begin{align*}
 P_\w( \s^{j,m}(N_j) > n_j \mu ) \leq e^{-\l \mu n_j/k_j} E e^{\l s^{j,m}(N_j)/k_j} \leq \left( e^{-\l \mu (1-\e)} ( e^{\l\mu'(1+\e)} + g_j(\l,\mu',\e, \d) ) \right)^{N_j}, 
\end{align*}
where in the second inequality we used that $n_j/k_j \leq (1-\e) N_j$ by the definitions of $k_j$ and $N_j$. 
The assumption that $\e < \frac{\mu-\mu'}{3\mu}$ and the fact that $g_j(\l,\mu',\e, \d) \ra 0$ as $j\ra\infty$ imply that $P-a.s.$ for $j$ sufficiently large,
\begin{equation}\label{embedded2}
 P_\w( \s^{j,m}(N_j) > n_j \mu ) \leq e^{-\l \e \mu N_j}, \qquad \forall m\in\mathcal{M}_j. 
\end{equation}
Applying \eqref{embedded1} and \eqref{embedded2} to \eqref{embedded} completes the proof of the lemma. 
\end{proof}
\begin{cor}\label{keycor}
For any $A<B$ and $\mu> \vp^{-1}$ and $\e < \frac{\mu-\vp^{-1}}{3\mu}$, 
\[
 \limsup_{j\ra\infty} \frac{1}{n_j^{1-1/s-\d}} \log P_\w( \exists m\in \mathcal{M}_j : T_{n_j + k_j}^{m k_j} \geq n_j \mu ) = -\infty, \qquad P-a.s.
\]
\end{cor}
\begin{proof}
 Choose $\mu'\in(\vp,\mu)$ so that the hypothesis of Lemma \ref{keylemma} are satisfied. Then, Lemma \ref{keylemma} implies that $P-a.s.$,
 \[
  P_\w( \exists m\in \mathcal{M}_j : T_{n_j + k_j}^{m k_j} \geq n_j \mu ) \leq |\mathcal{M}_j| \left( e^{-\theta \e n_j/2} + e^{-\l \e \mu n_j^{1-1/s-\d}} \right),
 \]
for all $j$ sufficiently large. Since the above is true for any $\l>0$ and since $|\mathcal{M}_j| = \bigo( n_j^{1-1/s-\d} )$, the statement of the Corollary follows. 
\end{proof}

\begin{proof}[\textbf{Proof of Proposition \ref{UQLDPslow}:}]$\left.\right.$\\
 This follows directly from Lemmas \ref{Timelem}, \ref{latticelem} and Corollary \ref{keycor}. 
\end{proof}

\end{section}

\begin{section}{Uniform LLN and Hydrodynamic Limits for RWRE}\label{Hydro}
In this section, we apply Proposition \ref{UQLDPslow} to prove Theorem \ref{ULLN} and we then use Theorem \ref{ULLN} to prove hydrodynamic limits for the system of RWRE.

\begin{proof}[\textbf{Proof of Theorem \ref{ULLN}:}]
In \cite{cgzLDP}, it was shown that an averaged large deviation principle holds for $X_n/n$ with a convex rate function $I(v)$. Moreover, it was shown in \cite{cgzLDP} that $I(v) > 0$ for any $v>\vp$. Therefore, averaged probabilities of speedups ($\{X_n > nv\}$ for some $v>\vp$) decay exponentially fast. Since we are only concerned about $(B-A)n^{\gamma+1}$ particles, a union bound and the Borel-Cantelli Lemma imply that 
\[
 \limsup_{n\ra\infty} \sup_{y\in(An,Bn], \; i\leq n^\gamma} \frac{X^{y,i}_n - y}{n} \leq \vp, \quad \P-a.s. 
\]
It remains only to show the corresponding lower bound.  
\begin{equation}\label{Uslowdowns}
 \liminf_{n\ra\infty} \inf_{y\in(An,Bn], \; i\leq n^\gamma} \frac{X^{y,i}_n - y}{n} \geq \vp, \quad \P-a.s. 
\end{equation}
We divide the proof of \eqref{Uslowdowns} into three cases: 
Strictly positive drifts, no negative drifts, and both positive and negative drifts. 

\noindent\textbf{Case I: Strictly positive drifts} - $P( \w_0 > \frac{1}{2} + \e ) = 1$ for some $\e>0$.\\
In the case of strictly positive drifts, it was shown in \cite{cgzLDP} that $I(v) > 0 $ for all $v<\vp$ as well. Therefore, averaged probabilities of slowdowns ($\{X_n < nv\}$ for some $v<\vp$) decay exponentially fast as well. Again, a union bound and the Borel-Cantelli Lemma can be used to obtain \eqref{Uslowdowns}. 

\noindent\textbf{Case II: No negative drifts} - $P( \w_0 < \frac{1}{2} ) = 0$ but $P( \w_0 \leq \frac{1}{2} + \e) > 0$ for all $\e>0$.\\
When $P( \w_0 \leq \frac{1}{2} + \e) > 0$ for all $\e>0$, it was shown in \cite{cgzLDP} that $I(v) = 0$ for all $v\in[0,\vp]$. Therefore, averaged probabilities of slowdowns decay subexponentially. However, if $P(\w_0 = \frac{1}{2} ) > 0$, then it was shown in \cite{dpzTE} that for any $v\in (0,\vp)$,
\[
 \limsup_{n\ra\infty} \frac{1}{n^{1/3}} \log \P( X_n < nv) < 0. 
\]
That is, averaged probabilities of slowdowns decay on an exponential scale like $e^{-C n^{1/3}}$. 
If $P( \w_0 \leq \frac{1}{2} + \e) > 0$ for all $\e>0$ but $P(\w_0 = \frac{1}{2} ) = 0$, then a coupling argument implies that changing all the sites with $\w_x \in (1/2, 1/2+\e)$ to have $\w_x = 1/2$ instead only increases the probability of a slowdown. Therefore, the averaged probabilities of slowdowns still decrease at least as fast as $e^{-C n^{1/3}}$, and again a union bound and the Borel-Cantelli Lemma imply \eqref{Uslowdowns}.

\noindent\textbf{Case III: Positive and Negative Drifts} - $P( \w_0 < \frac{1}{2} ) > 0$.\\
As mentioned previously, in this case the averaged probabilities for slowdowns decay only polynomially fast and so the above strategy of a union bound and the Borel-Cantelli Lemma no longer works with the averaged measure. 
The uniform ellipticity assumption in Assumption \ref{IIDasm} implies that $t\mapsto E_P[ \rho_0^t ]$ is a convex function of $t$ and is finite for all $t$. 
Also, the assumption in this case that $P(\w_0 < 1/2) > 0$ is equivalent to $P(\rho_0 > 1) > 0$, so that $E_P[\rho_0^t] \ra \infty$ as $t\ra\infty$.  
Therefore, $E_P[ \rho_0^s ] = 1$ for some $s>1$, and we may apply the uniform quenched large deviation estimates from Proposition \ref{UQLDPslow}. That is, if
\[
 \Omega_{v,\d} := \left\{ P_\w\left( \max_{y\in(An, Bn], \; i \leq n^\gamma } X^{y,i}_n - y < nv \right) \leq e^{-n^{1-1/s-\d}} \text{ for all large } n \right\},
\]
then $P(\Omega_{v,\d}) = 1$ for any $v\in(0,\vp)$ and $\d>0$. Thus, if $\Omega_\d:= \bigcap_{v \in (0,\vp)\cap \Q} \Omega_{v,\d}$, we have that $P(\Omega_\d) = 1$. Moreover, a union bound and the Borel-Cantelli Lemma imply that for any $\w\in \Omega_\d$, 
\[
 \liminf_{n\ra\infty} \inf_{y\in(An,Bn], \; i\leq n^\gamma} \frac{X^{y,i}_n - y}{n} \geq \vp, \quad P_\w-a.s. 
\]
Since $P(\Omega_\d) = 1$, this implies that \eqref{Uslowdowns} holds as well. 
\end{proof}
\begin{rem}
There are known conditions for multi-dimensional RWRE in uniformly elliptic i.i.d\ environments that imply a law of large numbers with $\lim_{n\ra\infty} X_n/n =: \vp \neq \mathbf{0}$ (for example, Kalikow's condition or conditions $(\mathbf{T})$ and $(\mathbf{T}')$ of Sznitman \cite{sSlowdown, sConditionT, sEffective}). Under these conditions, it is known that the probabilities of large deviations decay faster than any polynomial \cite{bSlowdown,sEffective}. 
Thus, it is easy to see that under these conditions, the multi-dimensional analogue of Theorem \ref{ULLN} holds. 
\end{rem}

We now show how the uniform law of large numbers in Theorem \ref{ULLN} can be used to prove the hydrodynamic limits for the system of RWRE as stated in Theorem \ref{strongweakhydro}. 

\begin{proof}[\textbf{Proof of Theorem \ref{strongweakhydro}:}]
For any $g\in\Cb$ we may choose $a<b$ such that the support of $g$ is contained in $(a,b]$ and the sums and the integrals in both \eqref{hydroIC} and \eqref{hydroConc} may be restricted to $x\in(aN,BN]$ and $y\in(a,b]$, respectively. 
Note that the representation of $\eta_{Nt}^N$ in \eqref{etat} implies that
\begin{align}
 \frac{1}{N} \sum_{x= \lfloor N a \rfloor + 1} ^{ \lfloor N b \rfloor } \eta^N_{Nt}(x) g(x/N) 
&= \frac{1}{N} \sum_{x= \lfloor N a \rfloor + 1} ^{ \lfloor N b \rfloor } \sum_{y\in \Z}\sum_{i=1}^{\eta_0^N(y)} \indd{X_{Nt}^{y,i} = x} g(x/N) \nonumber \\
&= \frac{1}{N} \sum_{y\in \Z}\sum_{i=1}^{\eta_0^N(y)} \sum_{x= \lfloor N a \rfloor + 1} ^{ \lfloor N b \rfloor } \indd{X_{Nt}^{y,i} = x} g(x/N) \nonumber \\
&= \frac{1}{N} \sum_{y\in \Z}\sum_{i=1}^{\eta_0^N(y)} \indd{X_{Nt}^{y,i} \in( Na, Nb] } g(X_{Nt}^{y,i}/N) \label{indicatorsums}
\end{align}
Let $A_t = a-t\vp$ and $B= b-t\vp$. 
The law of large numbers implies that the indicator function on the last line above should be almost the same as $\indd{y\in(A_tN,B_tN]}$. We can make this precise by applying Theorem \ref{ULLN}. 
Let 
\[
 E_{N,t,a,b,\d} := \left\{ | X_{Nt}^{y,i} - y - N t \vp | < N t \d , \quad \forall y\in(Na-Nt, Nb+Nt], \quad \forall i\leq \eta^N_0(y) \right\}.
\]
We claim that for any for any $\d>0$, $\P-a.s.$, the event $E_{N,t,a,b,\d}$ occurs for all $N$ large enough. First, note that if $C > \int_{a-t}^{b+t} \prof(y) dy$ then
the assumptions on the initial configurations imply that, $\P-a.s.$, for all $N$ sufficiently large, $\eta^N_0(y) < CN$ for all $y \in((a-t)N,(b+t)N]$. 
Indeed, 
\begin{align*}
\limsup_{N\ra\infty} \frac{1}{N} \max_{y\in((a-t) N, (b+t) N]} \eta^N_0(y)  &\leq \lim_{N\ra\infty} \frac{1}{N} \sum_{x=\lfloor N(a-t) \rfloor + 1}^{\lfloor N(b+t) \rfloor} \eta_0^N(y) \\
&= \int_{a-t}^{b+t} \prof(y) dy < C, \qquad \P-a.s.
\end{align*}
This in turn implies by Theorem \ref{ULLN} that, $\P-a.s$, for any $\d>0$ the event $E_{N,t,a,b,\d}$ occurs for all $N$ sufficiently large.

Now, on the event $E_{N,t,a,b,\d}$, 
\[
 y\in ((A_t+\d)N,(B_t-\d)N] \Longrightarrow X_{Nt}^{y,i} \in( Na, Nb],
\]
and
\[
 X_{Nt}^{y,i} \in( Na, Nb] \Longrightarrow y\in ((A_t-\d)N,(B_t+\d)N].
\]
Note that for the second implication above we used that the random walks are nearest neighbor random walks. 
Recalling \eqref{indicatorsums}, for any $\d>0$ and for all $N$ large enough, $\P-a.s.$,
\begin{align}
 \frac{1}{N} \sum_{x= \lfloor N a \rfloor + 1} ^{ \lfloor N b \rfloor } \eta^N_{Nt}(x) g(x/N) 
&= \frac{1}{N} \sum_{y = \lfloor NA_t \rfloor + 1}^{\lfloor NB_t \rfloor }\sum_{i=1}^{\eta_0^N(y)}  g(X_{Nt}^{y,i}/N) \nonumber \\
&\quad - \frac{1}{N} \sum_{y = \lfloor N(A_t-\d)\rfloor + 1}^{\lfloor N(A_t+\d) \rfloor} \sum_{i=1}^{\eta_0^N(y)} \mathbf{1}_{\{X_{Nt}^{y,i} \notin( Na, Nb] \}} g(X_{Nt}^{y,i}/N) \nonumber \\
&\quad - \frac{1}{N} \sum_{y = \lfloor N(B_t-\d)\rfloor + 1}^{\lfloor N(B_t + \d) \rfloor} \sum_{i=1}^{\eta_0^N(y)} \mathbf{1}_{\{X_{Nt}^{y,i} \notin( Na, Nb] \}} g(X_{Nt}^{y,i}/N). \label{Erewrite}
\end{align}
On the event $E_{N,t,a,b,\d}$, 
\begin{equation} \label{modcont}
 | g(X_{Nt}^{y,i}/N) - g(y/N+\vp t) | \leq \gamma(g,\d), \qquad \forall y\in((a-t)N, (b+t)N],
\end{equation}
where $\gamma(g,\d) = \sup \{ |g(x)-g(y)|: |x-y| < \d \}$ is the modulous of continuity of the function $g$. 
Recalling \eqref{Erewrite}, for any $\d>0$ and for all $N$ large enough, $\P-a.s.$,
\begin{align*}
 & \left| \frac{1}{N} \sum_{x= \lfloor N a \rfloor + 1} ^{ \lfloor N b \rfloor } \eta^N_{Nt}(x) g(x/N) - \int_a^b \prof(y - \vp t) g(y) dy \right| \\
&\quad \leq \left| \frac{1}{N} \sum_{y = \lfloor NA_t \rfloor + 1}^{\lfloor NB_t \rfloor } \eta_0^N(y)  g(y/N + \vp t) - \int_a^b \prof(y - \vp t) g(y) dy \right|  + \frac{1}{N} \sum_{y = \lfloor N A_t\rfloor + 1}^{\lfloor N B_t \rfloor } \eta_0^N(y) \gamma(g,\d) \\
&\qquad + \frac{1}{N} \sum_{y = \lfloor N(A_t-\d)\rfloor + 1}^{\lfloor N(A_t+\d) \rfloor} \eta_0^N(y) \|g\|_\infty + \frac{1}{N} \sum_{y = \lfloor N(B_t-\d)\rfloor + 1}^{\lfloor N(B_t + \d) \rfloor} \eta_0^N(y) \|g\|_\infty.
\end{align*}
Recalling that $A_t=a-t\vp$ and $B_t=b-t\vp$, we may do a change of variables to re-write
\[
 \int_a^b \prof(y - \vp t) g(y) dy = \int_{A_t}^{B_t} \prof(y) g(y+\vp t) dy. 
\]
Thus, the assumptions on the initial configurations imply that $\P-a.s.$,
\begin{align*}
& \limsup_{N\ra\infty} \left| \frac{1}{N} \sum_{x= \lfloor N a \rfloor + 1} ^{ \lfloor N b \rfloor } \eta^N_{Nt}(x) g(x/N) - \int_a^b \prof(y - \vp t) g(y) dy \right| \\
&\quad \leq \gamma(g,\d) \int_{A_t}^{B_t} \prof(y)dy + \|g\|_\infty \int_{A_t-\d}^{A_t + \d} \prof(y)dy + \|g\|_\infty \int_{B_t-\d}^{B_t + \d} \prof(y)dy. 
\end{align*} 

Since $g$ is uniformly continuous and $\prof$ is a bounded function, the right hand side can be made arbitrarily small by taking $\d\ra 0$. This proves \eqref{hydroConc} and thus finishes the strong version of the hydrodynamic limit. 
The proof of the weaker version of they hydrodynamic limit where \eqref{hydroIC} and \eqref{hydroConc} both hold in $P_\w$-probability is similar and is thus omitted. 
\end{proof}

\end{section}

\begin{section}{Stationary Distribution of the Particle Process}\label{StatDist}
We now change our focus away from the spatial scaling present in hydrodynamic limits and instead study the limiting distribution of particle configurations $\eta_n$ as $n\ra\infty$.
Recall that the initial configurations we are considering are such that given $\w$ the $\eta_0(x)$ are independent with distribution $\nu(\theta^x\w)$ where $\nu$ is a measurable function $\nu:\Omega\ra \Upsilon$ from the space of environments to the space of probability measures on $\Zp$. We begin with a couple of easy lemmas giving some properties of the system of RWRE under such initial conditions.


\begin{lem}\label{ergodic}
 Let $\nu: \Omega \ra \Upsilon$. Then the sequence $\{ \eta_0(x) \}_{x\in\Z}$ is ergodic under the measure $\P_{\nu}$. 
\end{lem}
\begin{proof}
 Let $F:\Upsilon \times [0,1] \ra \Zp$ be defined by
\[
 F(Q,u):= \inf \{ n \in \Zp : Q([0,n]) \geq u \}.
\]
$F$ is a measureable function, 
and if $U\sim U(0,1)$ then $F(Q,U)$ has distribution $Q$ for any $Q\in\Upsilon$. Now, let 
$\{U_x \}_{x\in \Z}$ be an \iid sequence of uniform $[0,1]$ random variables that are also independent of the random environment $\w = \{\w_x \}_{x\in\Z}$. Then, the joint sequence $\{ (\theta^x \w, U_x) \}_{x\in\Z}$ is ergodic. Finally, we can construct $\eta_0$ by letting
\[
 \eta_0(x) = F(\nu(\theta^x\w), U_x) . 
\]
Since $\eta_0$ can be constructed as a measurable function of an ergodic sequence which respects shifts of the original sequence, $\eta_0$ is ergodic as well. 
\end{proof}

\begin{lem} \label{statmean}
 Let $\nu: \Omega \ra \Upsilon$. Then $\E_\nu( \eta_0(0) ) = \E_\nu( \eta_n(0) )$ for all $n\in\N$. 
\end{lem}
\begin{proof}
For any environment $\w$ and $n\in\N$,
 \[
  E_{\w,\nu^\w}( \eta_n(0) ) = \sum_{x\in\Z} E_{\w,\nu^\w}( \eta_0(x) )P_\w( X_n^x = 0 ) = \sum_{x\in\Z} E_{\theta^x\w,\nu^{\theta^x\w}}( \eta_0(0) )P_{\theta^x\w}( X_n = -x ). 
 \]
Therefore, the shift invariance of $P$ implies that 
\begin{align*}
 \E_{\nu}( \eta_n(0) ) = E_P\left[ \sum_{x\in\Z} E_{\w,\nu^{\w}}( \eta_0(0) )P_{\w}( X_n = -x ) \right] 
= E_P\left[ E_{\w,\nu^{\w}}( \eta_0(0) ) \right] = \E_\nu ( \eta_0(0)). 
\end{align*}
\end{proof}

Recall the definitions of $\pi_\a: \Omega \ra \Upsilon$ and $f(\w)$ from \eqref{fdef}. 
The significance of the functions $\pi_\a$ is that the distributions $\pi_\a^\w$ are stationary under the (quenched) dynamics of the system of RWRE. 
\begin{lem}\label{StationaryDist}
Let $\pi_\a$ be defined as in \eqref{fdef} for some $\a>0$. Then, for $P-a.e.$ environment $\w$, $\pi_\a^{\w}$ is a stationary distribution for the sequence of random variables $\eta_n$. That is, if $\eta_0 \sim \pi_\a^{\w}$, then for any $n\in\N$, $\eta_n \sim \pi_\a^{\w}$ as well. 
\end{lem}
The analog of Lemma \ref{StationaryDist} for a system of continuous time RWRE was previously shown by Chayes and Liggett in \cite{clEPRE}. The proof for the discrete time model is essentially the same and is therefore ommitted. The key observation is that $f(\theta^x \w) = \w_{x-1}f(\theta^{x-1}\w) + (1-\w_{x+1}) f(\theta^{x+1}\w)$ for all $x\in\Z$, which can easily be checked by the definition of $f$ in \eqref{fdef}.

\begin{subsection}{The coupled process}
To complete the proof of Theorem \ref{uniquestat} we will introduce a coupling of two systems of RWRE in the same environment. 
Let $\nu, \s: \Omega \ra \Upsilon$ be measurable functions, and let $\eta_t$ and $\zeta_t$ be two systems of independent RWRE with initial configurations $\P_\nu$ and $\P_\s$ respectively. We will introduce a coupling of two systems of RWRE, $\eta_t$ and $\zeta_t$, that have marginal distributions $\P_\nu$ and $\P_\s$, respectively, and which maximizes the agreement between the two processes. We will follow the coupling procedure outlined in \cite{jsCPDE} (also in \cite{sTIESbook}). To this end, given $\w$, let $\eta_0$ and $\zeta_0$ be independent with distributions $\nu^\w$ and $\s^\w$, respectively.
Then, given the initial configurations $(\eta_0,\zeta_0)$, define
\begin{equation} \label{initialcouple}
 \xi_0(x) := \eta_0(x) \wedge \zeta_0(x), \quad \beta_0^+ = ( \eta_0(x)-\zeta_0(x) )^+, \quad\text{and}\quad \beta_0^- = ( \eta_0(x)-\zeta_0(x) )^- \;.
\end{equation}
$\xi_0(x)$ is the number of common particles at site $x$, and $\beta^{\pm}(x)$ is the excess number of $\eta_0$ or $\zeta_0$ particles at $x$. We will refer to the unmatched $\eta_0$ or $\zeta_0$ particles as $+$ or $-$ particles, respectively. (To make this rigorous, the particles in the initial configurations should be well ordered in some predetermined way and then the matchings at each site should be done with lowest labels matched first). At each time step the matched particles move together according to the law $P_\w$, while the excess $+$ and $-$ particles each move independently according to the law $P_\w$. After all the particles have moved we again match as many pairs of $+$ and $-$ particles at each site as possible. 

After $n$ time steps we denote the number of matched and unmatched $+$ or $-$ particles by
\begin{equation} \label{timecouple}
 \xi_n(x) := \eta_n(x) \wedge \zeta_n(x), \quad \beta_n^+ = ( \eta_n(x)-\zeta_n(x) )^+, \quad\text{and}\quad \beta_n^- = ( \eta_n(x)-\zeta_n(x) )^- \;.
\end{equation}
We will denote the quenched and averaged distributions of the coupled process $(\eta_n, \zeta_n)$ by $P_{\w,\nu^\w\times \s^\w}$ and $\P_{\nu\times \s}$, respectively. 
An easy adaptation of the proof of Lemma \ref{ergodic} shows that the joint sequence $\{ (\eta_0(x), \zeta_0(x)) \}_{x\in\Z}$ is ergodic under $\P_{\nu\times \s}$. Therefore, from \eqref{initialcouple} it is clear that the triple $\{(\xi_0(x), \beta_0^+(x), \beta_0^-(x))\}_{x\in\Z}$ is an ergodic sequence under $\P_{\nu\times\s}$ as well. 
\begin{lem}\label{ergodicbeta}
 Let $\nu,\s: \Omega \ra \Upsilon$. Then, for any $n\in\N$ the triple $\{(\xi_n(x), \beta_n^+(x), \beta_n^-(x))\}_{x\in\Z}$ is ergodic under the measure $\P_{\nu\times\s}$.
\end{lem}
\begin{proof}
We will give a more explicit construction of the coupling described above which makes the conclusion of the Lemma obvious. 
For each $x\in\Z$ and $n \geq 0$ let 
\[
 \Xi^x_n(\w) := \{ Y^{x,0}_n(j) ,\; Y^{x,+}_n(j),\;Y^{x,-}_n(j) :\; j \geq 1 \}
\]
be a collection of \iid random variables with distribution $\w_x \d_1 + (1-\w_x)\d_{-1}$, and let the collection $\Xi(\w):=\{\Xi^x_n(\w)\}_{x\in\Z, n \geq 0}$ be independent and independent of everything else as well. Assuming some determinstic rule for well-ordering the matched and unmatched $+$ or $-$ particles at each site, the random varibles $Y^{x,0}_n(j)$, $Y^{x,+}_n(j)$, and $Y^{x,-}_n(j)$ give the steps from time $n$ to $n+1$ of the $j^{th}$ matched and unmatched $+$ and $-$ particles at site $x$, respectively. 
%
Thus, we have that 
\begin{align*}
 \xi_{n+1}(x) &= \sum_{z\in\Z} \sum_{j=1}^{\xi_{n}(z)} \indd{ z + Y^{z,0}_n(j) = x} \\
&\qquad + \left( \sum_{z\in\Z} \sum_{j=1}^{\b^+_{n}(z)} \indd{ z + Y^{z,+}_n(j) = x} \right) \wedge \left( \sum_{z\in\Z} \sum_{j=1}^{\b^-_{n}(z)} \indd{ z + Y^{z,-}_n(j) = x} \right),
\end{align*}
\[
 \b_{n+1}^+(x) = \left( \sum_{z\in\Z} \sum_{j=1}^{\b^+_{n}(z)} \indd{ z + Y^{z,+}_n(j) = x}  - \sum_{z\in\Z} \sum_{j=1}^{\b^-_{n}(z)} \indd{ z + Y^{z,-}_n(j) = x} \right)^+,
\]
and
\[
 \b_{n+1}^-(x) = \left( \sum_{z\in\Z} \sum_{j=1}^{\b^+_{n}(z)} \indd{ z + Y^{z,+}_n(j) = x}  - \sum_{z\in\Z} \sum_{j=1}^{\b^-_{n}(z)} \indd{ z + Y^{z,-}_n(j) = x} \right)^-
\]
From the above construction of the coupled process, it is clear that for each $n$ there exists a measureable function $G_n$ such that 
\begin{equation}\label{Gndef}
 (\xi_n(x), \b^+_n(x), \b^-_n(x)) = G_n(\theta^x\eta_0, \theta^x\zeta_0, \Xi(\theta^x\w)),
\end{equation}
where the shift operator $\theta^x$ acts on configurations by $(\theta^x \eta_0)(y) = \eta_0(x+y)$. 
Finally, since $\Xi(\w)$ is independent of $(\eta_0,\zeta_0)$ (given $\w$), another simple adaptation of the proof of Lemma \ref{ergodic} gives that
$\left\{ (\theta^x\eta_0, \theta^x\zeta_0, \Xi(\theta^x\w) \right\}_{x\in\Z}$ is an ergodic sequence under the measure $\P_{\nu\times \s}$. This fact combined with \eqref{Gndef} finishes the proof. 
\end{proof}

\begin{cor} \label{monotone}
 Let $\nu,\s: \Omega \ra \Upsilon$. Then $\E_{\nu\times \s}( \b^+_n(x))$ and $\E_{\nu\times \s}( \b^-_n(x))$ do not depend on $x$ and are non-increasing in $n$. 
\end{cor}
\begin{proof}
 Since $\b^+_n$ and $\b^-_n$ are both ergodic (and thus stationary), the expectations do not depend on $x$. 
Also, two applications of Birkhoff's ergodic theorem and the fact that the number of unmatched $+$ particles in $[-m,m]$ at time $n+1$ is at most the number of unmatched $+$ particles in $[-m-1,m+1]$ at time $n$ imply that
\[
 \E_{\nu\times \s}( \b^+_{n+1}(0)) = \lim_{m\ra\infty} \frac{1}{2m+1} \sum_{x=-m}^m \b^+_{n+1}(x) \leq \lim_{m\ra\infty} \frac{1}{2m+1} \sum_{x=-m-1}^{m+1} \b^+_{n}(x) = \E_{\nu\times \s}( \b^+_{n}(0)) .
\]
The same argument shows that $\E_{\nu\times \s}( \b^-_n(x))$ is non-increasing as well. 
\end{proof}

\begin{prop}\label{couple}
 Let $\nu,\s: \Omega \ra \Upsilon$, and let $  \E_{\nu\times \s}(\zeta_0(0)) \leq \E_{\nu\times \s}(\eta_0(0)) < \infty$.  Then, 
\[
 \lim_{n\ra\infty} \E_{\nu\times \s} \left[ \b^-_n(0) \right] = \lim_{n\ra\infty} \E_{\nu\times\s} \left[ ( \eta_n(0)- \zeta_n(0))^- \right] = 0. 
\]
\end{prop}
\begin{proof}
 By Corollary \ref{monotone} it is enough to show that for any $\d>0$, $\E_{\nu\times \s} \left[ \b^-_n(0) \right] < \d$ for some $n$. Assume for contradiction that there exists a $\d>0$ such that $\E_{\nu\times \s} \left[ \b^-_n(0) \right] \geq \d$ for all $n$. Lemma \ref{statmean} and the assumptions of the proposition imply that 
\[
 \E_{\nu\times \s}(\b^+_n(0)) - \E_{\nu\times \s}(\b^-_n(0)) = \E_{\nu\times \s} \left[ \eta_n(0) - \zeta_n(0) \right] = \E_{\nu\times \s}(\eta_0(0)) - \E_{\nu\times \s}(\zeta_0(0)) \geq 0. 
\]
Therefore, $\E_{\nu\times \s} \left[ \b^+_n(0) \right] \geq \d$ for all $n$ as well. 

Given $\eta_0$ and $\zeta_0$, well-order the unmatched $+$ and $-$ particles and let $w_j^+(\cdot)$ and $w_j^-(\cdot)$ be the trajectories of the $j^{th}$ initially unmatched $+$ or minus particle, respectively. 
Then, for each $j$, denote the amount of time until $w_j^+$ is matched by $\tau_j^+ \in [1,\infty]$, and similalry let $\tau^-_j$ be the amount of time until $w_j^-$ is matched. If $\tau_j^{\pm} = \infty$, then the particle is said to be \emph{immortal}. Let 
\[
 \l_n^{\pm}(x) = \sum_j \indd{ w_j^\pm = x ,\; \tau_j^\pm > n},
\]
be the number of $\pm$ particles initially at site $x$ that are not matched after $n$ steps. A similar argument to the proof of Lemma \ref{ergodicbeta} implies that $\{(\l^+_n(x),\l^-_n(x))\}_{x\in\Z}$ is ergodic as well. In fact, because of certain periodicity issues that will arise later what we really need is that $\{ (\b^+_n(x), \b^-_n(x))\}_{x\in 2\Z}$ and $\{(\l^+_{n}(x),\l^-_n(x))\}_{x\in 2\Z}$ are ergodic. However, this also holds by essentially the same proof by noting that $\{ \theta^x \w \}_{x\in 2\Z}$ is an ergodic sequence since the environments are i.i.d. 
Two applications of Birkhoff's Ergodic Theorem imply that
\begin{align}
 \d \leq \E_{\nu\times\s}[ \b^\pm_n(0) ] 
&= \lim_{M\ra\infty} \frac{1}{2M+1} \sum_{x=-M}^M \b^\pm_n(2x) \nonumber \\
&\leq \lim_{M\ra\infty} \frac{1}{2M+1} \sum_{x=-M-n}^{M+n} \l^\pm_n(2x) 
= \E_{\nu\times\s}[ \l^\pm_n(0) ]. \label{avglonglife}
\end{align}

Let $\l^\pm_\infty(x) = \lim_{n\ra\infty} \l^\pm_n(x)$ be the number of immortal $\pm$ particles originally at $x$. Again, as was shown for $\l^\pm_n$ it can be shown that $\{(\l^+_\infty(x),\l^-_\infty(x))\}_{x\in 2\Z}$ is ergodic. Then \eqref{avglonglife} and the monotone convergence theorem imply that 
\[
 \lim_{M\ra\infty} \frac{1}{2M+1}\sum_{x=-M}^M \l_\infty^\pm(2x) = \E_{\nu\times \s} [\l_\infty^\pm(0)] = \lim_{n\ra\infty} \E_{\nu\times \s} [\l_n^\pm(0)] \geq \d.
\]
That is, there is a positive initial density of immortal $+$ and $-$ particles. 
Therefore, $\P_{\nu\times\s}-a.s.$, there exists an $M<\infty$ such that there exists at least one $+$ and $-$ immortal particle in $[-2M,2M]\cap 2\Z$. In particular, this implies that there exists an $M<\infty$, points $y,z\in [-2M,2M]\cap 2\Z$, and random walks starting at $y$ and $z$ that never meet. 
Since there are only finitely many particles initially in any finite interval, the following lemma gives a contradiction and thus finishes the proof of Proposition \ref{couple}. 
\begin{lem}\label{meeting}
 Let $y,z\in\Z$ be of the same parity (that is $z-y \in 2 \Z$). 
Then, $\P-a.s.$, two random walks in the same environment and starting at $y$ and $z$, respectively, must eventually meet. 
That is, $ \P\left( \exists n\geq 0:\; X^y_n = X^z_n \right) = 1$. 
\end{lem}

\end{proof}
\begin{proof}[\textbf{Proof of Lemma \ref{meeting}:}]
By the shift invariance of $P$, without loss of generality we may assume that $x=0$ and that $y<0$. Then, it is enough to prove that with $\P$-probability one, there is some site $z>0$ that the random walk started at $y<0$ reaches before the random walk started at $0$ does. That is, 
\[
 \P\left( \exists z > 0:\;  T^y_{z-y} < T^0_z \right) = 1. 
\]
Now, we may re-write $T^0_z = \sum_{i=1}^z \tau_i$ and $T^y_{z-y} = \sum_{i=y+1}^z \tilde{\tau}_i$, where $\tau_i$ and $\tilde{\tau}_i$ are the amount of time it takes to reach $i+1$ after first reaching $i$ for the walks $X^0_\cdot$ and $X^y_\cdot$, respectively. (That is, $\tau_i = T^0_i-T^0_{i-1}$ and $\tilde{\tau}_i = T^y_{i-y}-T^y_{i-y-1}$ for the $i$ appearing in each of the above sums.)
Note that for any $i\geq 1$, given the environment $\w$ (that is under $P_\w$) $\tau_i$ and $\tilde{\tau}_i$ are independent and have the same distribution. 
Now, for any $z>0$ we have that 
\[
 T^y_{z-y} < T^0_z \iff \sum_{i=y+1}^z \tilde{\tau}_i < \sum_{i=1}^z \tau_i \iff \sum_{i=y+1}^0 \tilde{\tau}_i < \sum_{i=1}^z \left( \tau_i - \tilde{\tau}_i \right)
\]
Therefore, we wish to show that $\P\left( \exists z>0: \; \sum_{i=y+1}^0 \tilde{\tau}_i < \sum_{i=1}^z \left( \tau_i - \tilde{\tau}_i \right) \right) = 1$. Since $\sum_{i=y+1}^0 \tilde{\tau}_i$ is finite, $\P-a.s.$, it is enough to show that 
\begin{equation}\label{limsuptau}
 \P\left( \sup_{z>0} \sum_{i=1}^z (\tau_i - \tilde{\tau}_i) = \infty \right) = 1.
\end{equation}
It is known that the sequence $\{\tau_i\}_{i\in\Z_+}$ is ergodic under the averaged measure $\P$ (see \cite{sRWRE}). This same argument shows that $\{(\tau_i,\tilde{\tau}_i) \}_{i\in\Z_+}$ is ergodic as well. In particular, this implies that $\tau_i-\tilde{\tau}_i$ is an ergodic sequence. Since the event in \eqref{limsuptau} is shift invariant, it is enough to prove that the right-hand side of \eqref{limsuptau} is non-zero.

Let $i_k$ be the sequence of indices where the $\tau_i$ and $\tilde{\tau}_i$ are different. That is,
\[
 i_1 = \min \{i\geq 1:\, \tau_i \neq \tilde{\tau}_i \}, \qquad i_{k+1} = \inf \{ i> i_k: \, \tau_i \neq \tilde{\tau}_i \}, \quad k\geq 1. 
\]
Define for any integers $l,M\geq 1$ the event 
\[
 E_{l,M} := \bigcap_{k=l}^{l+2M} \{ \tau_{i_k} > \tilde{\tau}_{i_k} \}. 
\]
Since $\tau_i$ and $\tilde{\tau}_i$ are independent and identically distributed under $P_\w$, given that they are different they are each equally likely to be the larger than the other. That is, $P_\w( \tau_{i_k} > \tilde{\tau}_{i_k} ) = 1/2$ for all $k\geq 1$. Also, $\{ \tau_i - \tilde{\tau}_i \}_{i\geq 1}$ is an independent sequence of random variables under $P_\w$, and thus by comparison with an infinite sequence of fair coin tosses we obtain that 
for any $M\geq 1$, 
\begin{equation}\label{Elmoccurs}
 P_\w\left( \bigcup_{l=1}^\infty E_{l,M} \right) = 1, \qquad P-a.s.
\end{equation}

Note that the event $E_{l,M}$ implies that $\sup_{z>0} \left| \sum_{i=1}^z (\tau_{i} - \tilde{\tau}_i)  \right| > M$.
Indeed, for $z=i_{l+2M}$ we obtain that
\[
 \sum_{i=1}^{i_{l+2M}} (\tau_{i} - \tilde{\tau}_i) = \sum_{k=1}^{l+2M}  (\tau_{i_k} - \tilde{\tau}_{i_k}) = \sum_{k=1}^{l-1}  (\tau_{i_k} - \tilde{\tau}_{i_k}) + \sum_{k=l}^{l+2M}  (\tau_{i_k} - \tilde{\tau}_{i_k}).
\]
$E_{l,M}$ implies that the second sum on the right above is at least $2M+1$. Therefore, either the first sum is less than $-M$ or the first and second sum together on the right are greater than $M$. 
Therefore, since \eqref{Elmoccurs} holds for any $M\geq 1$ we obtain that
\[
 1 = \P\left( \sup_{z > 0} \left|\sum_{i=1}^z (\tau_i - \tilde{\tau}_i) \right| = \infty \right) \leq  \P\left( \sup_{z > 0} \sum_{i=1}^z (\tau_i - \tilde{\tau}_i)  = \infty \right) + \P\left( \inf_{z > 0} \sum_{i=1}^z (\tau_i - \tilde{\tau}_i)  = - \infty \right). 
\]
Since the $\tau_i$ and $\tilde{\tau}_i$ are identically distributed this implies that $\P\left( \sup_{z > 0} \sum_{i=1}^z (\tau_i - \tilde{\tau}_i)  = \infty \right) \geq \frac{1}{2}$. However, as noted above, the ergodicity of $\tau_i-\tilde{\tau}_i$ implies that this last probability is in fact equal to $1$. 
This completes the proof of \eqref{limsuptau} and thus also the proof of the Lemma.
\end{proof}

\end{subsection}

We now return to the proof of Theorem \ref{uniquestat}. 
\begin{proof}[\textbf{Proof of Theorem \ref{uniquestat}:}]
Let $(\eta_n,\zeta_n)$ be the coupled process as described above with law $\P_{\nu \times \pi_\a}$, where $\a=\vp \E_\nu( \zeta_0(0) )$ so that $\E_{\nu}(\eta_0(0)) = \E_{\pi_\a}( \zeta_0(0) )$. Proposition \ref{couple} then implies that 
\[
 \lim_{n\ra\infty} \P_{\nu \times \pi_\a} ( \eta_n(0) \neq \zeta_n(0) ) \leq \lim_{n\ra\infty} \E_{\nu \times \pi_\a} | \eta_n(0) - \zeta_n(0) | = 0. 
\]
Therefore, for any cylinder set $E\subset (\Zp)^\Z$, 
$
 \lim_{n\ra\infty} \left| \P_{\nu} (\eta_n \in E ) - \P_{\pi_\a} (\zeta_n \in E) \right| = 0. 
$
However, Lemma \ref{StationaryDist} implies that $\P_{\pi_\a} (\zeta_n \in E) = \P_{\pi_\a}(\zeta_0 \in E)$, and thus 
\begin{equation}\nonumber
 \lim_{n\ra\infty} \P_\nu(\eta_n \in E) = \P_{\pi_\a}(\zeta_0 \in E). 
\end{equation}

It remains only to show that the sequence $\{\zeta_n\}_{n\geq 0}$ is tight. Indeed, for any positive integer $M$ let 
\[
 K_{M}:= \{ z=(z_x)_{x\in\Z} \in (\Zp)^\Z: z_x < M^2, \; \forall x\in[-M,M] \}.
\]
Then, $K_{M}$ is a compact subset of $(\Zp)^\Z$. Moreover,
\begin{align} \label{tightness}
 \P_{\nu}( \zeta_n \notin K_{M} ) \leq \sum_{x=-M}^M \P_\nu(\zeta_n(x) \geq M^2 ) \leq \E_\nu (\zeta_0(0))\frac{2M+1}{M^2} . 
\end{align}
where the last equality is from Chebychev's Inequality, Lemma \ref{statmean}, and the fact that $\zeta_0(x)$ is a stationary sequence under $\P_\nu$. Therefore, since the expectation on the right is finite by assumption, $\lim_{M\ra\infty} \sup_n \P_{\nu}( \zeta_n \notin K_{M} ) = 0$.
\end{proof}

\end{section}

\def\cprime{$'$}


\end{document}